\documentclass[10pt,a4paper,reqno]{amsart}
\usepackage{graphicx}
\usepackage[linktocpage=true,colorlinks,citecolor=blue,linkcolor=blue,urlcolor=black,pagebackref]{hyperref}
\usepackage{amssymb}
\usepackage{cite}
\usepackage{amsmath}
\usepackage{latexsym}
\usepackage{amscd}
\usepackage{amsthm}
\usepackage{mathrsfs}
\usepackage{url}
\usepackage{amsmath,mathtools}
\usepackage{cleveref}
\usepackage{hyperref}

\newtheorem{thm}{Theorem}[section]

\newtheorem{lemma}[thm]{Lemma}
\newtheorem{prop}[thm]{Proposition}

\newtheorem{ques}[thm]{Question}

\theoremstyle{definition}
\newtheorem{defn}[thm]{Definition}
\newtheorem{rem}[thm]{Remark}
\newtheorem*{ack}{Acknowledgments}
\numberwithin{equation}{section}
\newcommand{\RNum}[1]{\uppercase\expandafter{\romannumeral #1\relax}}

\def\td{\tilde}

\def\pt{\partial}
\def\del{\nabla}
\def\G2{\GG_2}
\def\g2{\varphi}

\def\ol{\overline}

\DeclareMathOperator\GG{G}

\DeclareMathOperator\Div{div}

\crefformat{section}{\S#2#1#3} 
\crefformat{subsection}{\S#2#1#3}
\crefformat{subsubsection}{\S#2#1#3}

\begin{document}
\title[Minimal Hypersurfaces in nearly $\G2$ Manifolds]{Minimal Hypersurfaces in nearly $\G2$ Manifolds}
\author{Shubham Dwivedi}
\address{Department of Pure Mathematics, University of Waterloo, Waterloo, ON N2L3G1}
\email{\href{mailto:s2dwived@uwaterloo.ca}{s2dwived@uwaterloo.ca}}
\date{\today}

\begin{abstract}
We study hypersurfaces in a nearly $\G2$ manifold. We define various quantities associated to such a hypersurface using the $\G2$ structure of the ambient manifold and prove several relationships between them. In particular, we give a necessary and sufficient condition for a hypersurface with an almost complex structure induced from the $\G2$ structure of the ambient manifold, to be nearly K\"ahler. Then using the nearly $\G2$ structure on the round sphere $S^7$, we prove that for a compact minimal hypersurface $M^6$ of constant scalar curvature in $S^7$ with the shape operator $A$ satisfying $|A|^2>6$, there exists an eigenvalue $\lambda >12$ of the Laplace operator on $M$ such that $|A|^2=\lambda - 6$, thus giving the next discrete value of $|A|^2$ greater than $0$ and $6$, thus generalizing the result of \cite{deshmukh1} about nearly K\"ahler $S^6$.
\end{abstract}

\maketitle

\tableofcontents{}

\section{Introduction}\label{sec-intro}

Let $(\ol{M}^7, \ol{g})$ be a Riemannian manifold with a vector cross product $B$. Then they induce a $\G2$ structure $\g2$ on $\ol{M}$, i.e., $\g2$ is a $3$-form  which is non degenerate in some sense (see \cref{sec-prelims} for precise definitions). Let $M^6$ be a hypersurface of $\ol{M}$ with the induced metric $g$ from $\ol{g}$ and denote by $N$ the unit normal vector field of $M$ in $\ol{M}$. If we define $\xi : TM\rightarrow TM$ by $\xi(X)=B(N,X)$, where $X\in \Gamma(TM)$ and $B$ is the vector cross product, then $\xi$ is a metric compatible almost complex structure on $M$ (cf. Proposition \ref{prop3.1}). More generally, if $(L,g,J)$ is an almost Hermitian manifold with an almost complex structure $J$, then we have the following

\begin{defn}\label{nearlykahlerdefn}
Let ($L,g,J$) be an almost Hermitian manifold with an almost complex structure $J$. Then $L$ is called a \emph{nearly K\"ahler} manifold if $\del J$ is a skew-symmetric tensor, i.e.,
\begin{equation}\label{nearlykahlercondition}
(\del_XJ)X=0 , \hspace{0.3cm} \forall \ X\in \Gamma(TM)
\end{equation}
\end{defn}

So a natural question is to find conditions on the oriented hypersurface $M$ so that with respect to the almost complex structure $\xi$, $(M,g,\xi)$ is a nearly K\"ahler manifold. Our first result is a characterization of nearly K\"ahler hypersurfaces of manifolds with a $nearly$ $\G2$ structure (see \cref{sec-prelims} for definition). In \cref{sec3}, we prove the following (cf. Theorem \ref{maintheorem1})

\begin{thm}\label{mainthm1}
Let $M$ be an oriented hypersurface of a nearly $\G2$ manifold $(\ol{M},\g2)$. Then $(M,g,\xi)$ is a nearly K\"ahler structure if and only if $M$ is totally umbilic, i.e.,  for all $X\in \Gamma(TM)$
\begin{equation}
AX=\alpha X
\end{equation}
where $A$ is the shape operator of $M$ in $\ol{M}$ and $\alpha \in C^{\infty}(TM)$.
\end{thm}
We note that Theorem \ref{mainthm1} was already proved in \cite[Theorem 4.8]{gray-VCP}. However, for our proof of Theorem \ref{mainthm1}, we define new quantities related to a manifold with a nearly $\G2$ structure which have analogs in the study of manifolds with a nearly K\"ahler structure and which we hope will be of further use in the study of submanifolds of manifolds with a nearly $\G2$ structure. 

\vspace{0.2cm}

In a different but related direction, suppose $M^n$ is a closed $minimal$ hypersurface of \emph{constant scalar curvature} in the unit sphere $S^{n+1}$ and let $A$ be its shape operator. A famous rigidity theorem due to the combined works of Simons \cite{simons}, Lawson\cite{lawsonminimal} and Chern-doCarmo-Kobayashi \cite{cherndocarmo} states that if $|A|^2 \leq n$ then $|A|^2=0$ or $|A|^2=n$, where $|A|^2$ is the squared length of the shape operator. If $|A|^2=0$, then $M$ is isometric to the totally geodesic equatorial sphere $S^n$ in $S^{n+1}$ and if $|A|^2=n$, then $M$ is isometric to the Clifford torus $S^k\Big(\sqrt{\frac{k}{n}}\Big) \times S^{n-k}\Big(\sqrt{\frac{n-k}{n}}\Big)$. Following on his study of subsequent gaps for the scalar curvature of such hypersurfaces $M$, Chern asked the following question (cf. \cite[pg.693]{yaulist})

\begin{ques}
\text{\bf{[Chern]}} Consider the set of all compact minimal hypersurfaces in $S^{n+1}$ with constant scalar curvature. Think of the scalar curvature as a function on this set. Is the image of the scalar curvature function a discrete set of positive numbers ?
\end{ques}
Since for any minimal hypersurface $M^n$ with scalar curvature $S$ in $S^{n+1}$, $S=n(n-1)-|A|^2$ (cf. \eqref{hyperscalar} in \cref{sec-prelims}), the above question asks whether the set of $|A|^2$ for such hypersurfaces $M$ is a discrete set.

The first two values of $|A|^2$ are known to be $0$ and $n$. For the third value of $|A|^2$, Peng and Terng \cite{peng-terng1} proved that if $|A|^2>n$, then there exists a positive constant $\delta(n)$ such that $|A|^2>n+\delta(n)$. Also, for $n=3$ they proved that $|A|^2\geq 6$ and they conjectured that the third value of $|A|^2$ should be equal to $2n$. Yang and Cheng in \cite{chengyang1} improved the constant $\delta(n)$ by proving that $\delta(n)>\frac {2}{7}n-\frac{9}{14}$ and in \cite{chengyang2} they further improved this result by proving that if $|A|^2>n$ then $|A|^2>\frac{1}{3}(4n+1)$. In \cite{deshmukh1}, Deshmukh used the nearly K\"ahler structure on $S^6$ to prove the following theorem

\begin{thm}\label{deshmukh}
\text{\bf{[Deshmukh, \cite{deshmukh1}}]} Let $M$ be a compact minimal hypersurface of constant scalar curvature in the unit sphere $S^6$. If the shape operator $A$ of $M$ satisfies $|A|^2>5$, then there exists an eigenvalue $\lambda >10$ of the Laplace operator on $M$ satisfying $|A|^2=\lambda -5$.
\end{thm}

The round unit sphere $S^7$ has a nearly $\G2$ structure, so a natural question is that whether we can say anything about the third value of $|A|^2$ for compact minimal hypersurfaces with constant scalar curvature in $S^7$ by using the nearly $\G2$ structure on it. Our next result is an analog of Theorem \ref{deshmukh} for minimal hypersurfaces with constant scalar curvature in $S^7$. More precisely we prove the following

\begin{thm}\label{mainthm2}
Let $M^6$ be a compact minimal hypersurface of constant scalar curvature in the unit sphere $S^7$. If the shape operator $A$ of $M$ satisfies $|A|^2>6$, then there exists an eigenvalue $\lambda >12$ of the Laplace operator on $M$ such that $|A|^2=\lambda - 6$.
\end{thm}

\medskip

This puts a restriction on possible examples of compact minimal hypersurfaces of constant scalar curvature in $S^7$ which have $|A|^2>6$ as they must have an eigenvalue $\lambda>12$ of the Laplacian operator such that $|A|^2=\lambda-6$.

\medskip
 
The paper is organized as follows. In \cref{sec-prelims} we discuss preliminaries on vector cross products on manifolds and then proceed to define manifolds with $\G2$ structures. We describe the intrinsic torsion forms of a $\G2$ structure and use them to define a nearly $\G2$ structure. We also discuss some notions from the geometry of submanifolds. In \cref{sec3} we start by defining several quantities associated to a hypersurface of a nearly $\G2$ manifold and then prove various relations among them. Using that we prove Theorem \ref{mainthm1}. We note that several of the results in \cref{sec3} are already known. However, we prove them using our notations to make the paper self contained and give reference of the original result accordingly. Finally in \cref{sec-thm2}, we prove Theorem \ref{mainthm2}.

\begin{ack} 
The author would like to thank his advisor Spiro Karigiannis for innumerable discussions related to the paper and for his constant encouragement and advice. The author would also like to thank Ragini Singhal for fruitful conversations related to the paper. Finally, the author would like to thank the anonymous referee for making several useful suggestions for the paper and for suggesting some references.  
\end{ack}

\section{Preliminaries }\label{sec-prelims}

\subsection{Manifolds with Vector Cross Product}
Let $(M^n, g)$ be a Riemannian manifold. An $r$-fold $vector \ cross \ product$ (VCP, for short) is an alternating $r$-linear smooth map 
\begin{equation}\label{B defn}
B:\underbrace{TM\times TM\times \cdots \times TM}_{r}\rightarrow TM
\end{equation}

satisfying the following conditions 
\[ \begin{cases} 
       g(B(v_1,...,v_r), v_i) =0, & 1\leq i\leq r  \\
      \|B(v_1,...,v_r)\|^2 = \|v_1\wedge  \cdots \wedge v_r\|^2 \\
     
   \end{cases}
\]

for any $v_i \in TM$.  \\
Such a cross product gives rise to a $(r+1)$ differential form $\phi$ defined as 
\begin{equation*}
\phi(v_1,v_2,...,v_{r+1})=g(B(v_1,...,v_r),v_{r+1})
\end{equation*}

The VCP is called $parallel / closed$ if and only if the corresponding differential form is $parallel /closed$.

\medskip 

Cross products on real vector spaces were classified by Brown and Gray in \cite{brown-gray} and global cross products on manifolds are discussed in Gray \cite{gray-VCP}. 
The classification of VCPs on a real vector space $V$ with a positive definite inner product $g$ is as follows:

\begin{enumerate}
\item $r=1$. Then a $1$-fold VCP $B$ on $V$ is equivalent to an almost complex structure on $V$, i.e., $B^2=-I$ on $V$. The associated VCP form is the K\"ahler form $\omega$.

\item $r=n-1$, where $n$ is the dimension of $V$. An $(n-1)$-fold VCP $B$ on $V$ is the $Hodge \ star$ operator $\star$ given by $g$ on $\Lambda^{n-1}V$ and the VCP form of degree $n$ is the volume form on $V$. Thus $B$ is equivalent to an orientation. 

\item $r=2$. A $2$-fold VCP $B$ on $\mathbb{R}^7$ is a cross product defined as $B(u,v)=\text{Im}(u.v)$, for $u,v$ in $\mathbb{R}^7\cong\text{Im}\mathbb{O}$, the set of imaginary octonions. Here $.$ is the octonionic multiplication. For coordinates $\{x_1,...,x_7\}$ on $\text{Im}\mathbb{O}$, the VCP form of degree $3$ can be written as follows

\begin{align}\label{fundamental3form}
\g2_0 &= dx^{123}-dx^{167}-dx^{527}-dx^{563}+dx^{415}+dx^{426}+dx^{437}
\end{align}

where $dx^{ijk}=dx^i\wedge dx^{j}\wedge dx^{k}$. Bryant \cite{bryant1} showed that the group of linear transformations of $\mathbb{O}$ which preserves $\g2_0$ also preserves $g$ and $B$ and is the exceptional lie group $G_2$ which is also the automorphism group of $\mathbb{O}$.

\item $r=3$. A $3$-fold VCP $B$ on $\mathbb{R}^8$ is a cross product defined as $B(u,v,w)=\frac 12(u(\bar{v}w)-w(\bar{v}u))$ for any $u,v,w$ in $\mathbb{R}^8 \cong \mathbb{O}$. For coordinates $x_1,...,x_8$ on $\mathbb{O}$, the VCP form of degree $4$ can be written as 

\begin{align*}
\Omega_0 &= -dx^{1234}-dx^{5678}-(dx^{21+dx^{34}})(dx^{65}+dx^{78})-(dx^{31}+dx^{42})(dx^{75}+dx^{86}) \nonumber \\
& \quad-(dx^{41}+dx^{23})(dx^{85}+dx^{67}) 
\end{align*}

Bryant \cite{bryant1} showed that the group of linear transformations of $\mathbb{O}$ preserving $\Omega_0$ also preserves $g$ and $B$ and is the group $Spin(7)\subset SO(8)$.

\end{enumerate}

\medskip 

In particular, the cross product on a $7$-dimensional vector space induced by the octonionic multiplication has the following properties. For all $u,v,w\in (V,g)$

\begin{align} 
 g(u,B(v,w))&=g(B(u,v),w) \label{cp1} \\
B(u,B(u,v))&=-g(u,u)v+g(u,v)u \label{malcev}\\
B(u,B(v,w))+B(v,B(u,w))&=g(u,w)v+g(v,w)u-2g(u,v)w \label{cp2}
\end{align}

\subsection{Manifolds with $\G2$ and nearly $\G2$ structure}
In this section, we review the concept of a $\G2$ structure on a manifold and the associated decomposition of the space of forms, mainly based on \cite{bryantremarks}, \cite{joycebook} and \cite{kar05}. In particular, we use the sign convention used in \cite{kar05}(which is opposite to that used in \cite{bryantremarks}). We describe the Torsion of a $\G2$ structure and the four torsion tensors. Finally, we define $nearly \ \G2$ structures. \\

Let $\{e_1,...,e_7\}$ be the standard basis of $\mathbb{R}^7$ and  $\{e^1,...,e^7\}$ be the dual basis. The VCP form $\g2_0$ induced by the vector cross product on $\mathbb{R}^7$ is described in \eqref{fundamental3form}. The group $\G2$ preserves $\g2_0$ and it also preserves the metric and orientation for which $\{e_1,...,e_7\}$ is an oriented orthonormal basis. If $\star_{\g2_0}$ denotes the Hodge star determined by the metric and the orientation, then $\G2$ preserves the $4$-form

\begin{align}\label{fundamental4form}
\psi_0 = \star_{\g2_0}\g2_0 & =dx^{4567}-dx^{4523}-dx^{4163}-dx^{4127}+dx^{2637}+dx^{1537}+dx^{1526}
\end{align}

where $dx^{ijkl}=dx^i\wedge dx^j\wedge dx^k\wedge dx^l$. \\

Let $M$ be a $7$-manifold. For $x\in M$, we denote by

\begin{equation*}
\Lambda^3_{pos}(M)_x = \{\g2_x\in \Lambda^3T^*_xM \mid \text{$\exists$ isomorphism $\rho:T_xM \rightarrow \mathbb{R}^7$}, \rho^*\g2_0=\g2_x\}
\end{equation*}

The bundle $\Lambda^3_{pos}(M) = \sqcup_{x\in M}\Lambda^3_{pos}(M)_x $ is an open subbundle of $\Lambda^3T^*M$ as $\Lambda^3_{pos}(M)_x\cong \text{GL}(7,\mathbb{R}) / \G2$. A section $\g2$ of $\Lambda^3_{pos}(M)$ is called a positive $3$-form on $M$ and the space of positive $3$-forms on $M$ is denoted by $\Omega^3_{pos}(M)$. Such a $\g2$ is also called a $\G2$ structure on $M$. A $\G2$ structure exists on $M$ if and only if the manifold is orientable and spin, which is equivalent to the vanishing of the first and second Stiefel-Whitney classes.
 
 \qquad A $\G2$ structure induces a unique metric and orientation. For a $3$-form $\g2$, we define 
 
 \begin{equation*}
 S_{\g2}(u,v)=-\frac 16 (u\lrcorner \g2)\wedge(v\lrcorner \g2)\wedge \g2
 \end{equation*}
for $u,v$ tangent vectors on $M$, which is a $\Omega^7(M)$-valued bilinear form. The $3$-form $\g2$ is a positive $3$-form if and only if $S_\g2$ is a tensor product of a positive definite bilinear form and a nowhere vanishing $7$-form which defines a unique metric $g$ with volume form $vol_g$ by

\begin{equation*}
g(u,v)vol_g=S_{\g2}(u,v).
\end{equation*}

The metric and orientation determines the Hodge star operator $\star_{\g2}$ and we define $\psi=\star_{\g2}\g2$.

A $\G2$ structure on $M$ induces a splitting of the spaces of differential forms on $M$ into irreducible $\G2$ representations.  The space of $2$-forms $\Omega^2(M)$ and $3$-forms $\Omega^3(M)$ decompose as 
\begin{align}
\Omega^2(M)&=\Omega^2_7(M)\oplus \Omega^2_{14}(M) \label{decomp1}\\
\Omega^3(M)&=\Omega^3_1(M)\oplus \Omega^3_7(M)\oplus \Omega^3_{27}(M) \label{decomp2}
\end{align}
 where $\Omega^k_l$ has pointwise dimension $l$. More precisely, we have the following description of the space of forms :
 
 \begin{align} \label{2formsdecomposition}
 \Omega^2_7(M) &=\{X\lrcorner \g2\mid X\in \Gamma(TM)\} = \{\beta \in \Omega^2(M)\mid \star(\g2\wedge \beta)=-2\beta\} \\
 \Omega^2_{14}(M) &=\{\beta \in \Omega^2(M)\mid \beta \wedge \psi =0 \} = \{\beta\in \Omega^2(M)\mid \star(\g2\wedge \beta)=\beta\} \\
 \end{align}
 
 and
 
 \begin{align}\label{3formsdecomposition}
  \Omega^3_1(M) & = \{f\g2 \mid f\in C^{\infty}(M) \} \\
   \Omega^3_7(M) & =\{X\lrcorner \psi \mid X\in \Gamma(TM) \} \\
    \Omega^3_{27}(M) & = \{ \gamma \in \Omega^3(M) \mid \gamma \wedge \g2= 0 = \gamma \wedge \psi \} \nonumber \\
    &= \{h_{ij}g^{jl}dx^i\wedge (\frac{\pt}{\pt x^l} \lrcorner \g2) \mid h_{ij}=h_{ji}, g^{ij}h_{ij} = 0 \}\label{three27}
 \end{align}
 in local coordinates $\{x^1,...,x^7\}$ on $M$. In \eqref{three27}, $h$ is a symmetric $2$ tensor. The decompositions of $\Omega^4(M)=\Omega^4_1(M)\oplus \Omega^4_7(M)\oplus \Omega^4_{27}(M)$ and $\Omega^5(M)=\Omega^5_7(M)\oplus \Omega^5_{14}(M)$ are obtained by taking the Hodge star of \eqref{decomp2} and \eqref{decomp1} respectively. The contractions  between  $\g2$ and $\psi$ in index notation (see \cite{bryantremarks} or \cite{kar05} for more details) are as follows

\begin{align}\label{contractions1}
\g2_{ijk}\g2_{abc}g^{kc} & = g_{ia}g_{jb}-g_{ib}g_{ja}-\psi_{ijab} \\
\g2_{ijk}\psi_{abcd}g^{kd} & = g_{ia}\g2_{jbc}+g_{ib}\g2_{ajc}+g_{ic}\g2_{abj} \nonumber \\
& \quad -g_{ja}\g2_{ibc}-g_{jb}\g2_{aic}-g_{jc}\g2_{abi} \label{contractions2} \\
\psi_{ijkl}\psi_{abcd}g^{jb}g^{kc}g^{ld} &= 24 g_{ia}
\end{align}

Given a $\G2$ structure $\g2$ on $M$, we can decompose $d\g2$ and $d\psi$ according to \eqref{decomp1} and \eqref{decomp2}. This defines \emph{torsion forms}, which are unique differential forms $\tau_0 \in \Omega^0(M)$, $\tau_1 \in \Omega^1(M)$, $\tau_2 \in \Omega^2_{14}(M)$ and $\tau_3 \in \Omega^3_{27}(M)$ such that (see \cite{kar05})

\begin{align}\label{torsionforms}
d\g2 &= \tau_0\psi + 3\tau_1\wedge \g2 + \star_{\g2}\tau_3 \\
d\psi &= 4\tau_1\wedge \psi + \star_{\g2} \tau_2 
\end{align}

The full torsion tensor $T$ of a $\G2$ structure is a $2$-tensor satisfying 

\begin{align}
\del_i\g2_{jkl} &= T_{im}g^{mp}\psi_{pjkl} \label{torsion1} \\
T_{lm} &=\frac {1}{24}(\del_l\g2_{abc})\psi_{mijk}g^{ia}g^{jb}g^{kc} \label{torsion2}\\
\del_m\psi_{ijkl} &= -T_{mi}\g2_{jkl}+T_{mj}\g2_{ikl}-T_{mk}\g2_{ijl}+T_{ml}\g2_{ijk} \label{torsion3}
\end{align}

The full torsion $T$ is related to the torsion forms by (see \cite{kar05}) 

\begin{equation}\label{torsionrel}
T_{lm}=\frac{\tau_0}{4}g_{lm}-(\tau_3)_{lm}+(\tau_1)_{lm}-\frac{1}{2}(\tau_2)_{lm}
\end{equation}

\begin{rem}
Since the space $\Omega^2_{7}$ is isomorphic to the space of vector fields and hence to the space of $1$-forms so in \eqref{torsionrel}, we are viewing $\tau_1$ as an element of $\Omega^2_{7}$ which justifies the expression $(\tau_1)_{lm}$. See \cite{kar05} for more details.
\end{rem}

\begin{defn}
A $\G2$ structure $\g2$ is called torsion free if $\del \g2 =0$ or equivalently $T=0$.
\end{defn}

A manifold $(M,\g2)$ with a $\G2$ structure $\g2$ is called a $\G2$ manifold if it is torsion-free.

 We can now define nearly $\G2$ structure.
 
 \begin{defn}\label{nearlyg2defn}
 A $\G2$ structure $\g2$ is a nearly $\G2$ structure if $\tau_0$ is the only nonvanishing component of the torsion, i.e.,  $d\g2=\tau_0\psi$ and $d\psi=0$
 \end{defn}

In this case, we see from \eqref{torsionrel} that $T_{ij}=\frac{\tau_0}{4}g_{ij}$. 

\begin{rem}
If $\g2$ is a nearly $\G2$ structure on $M$ then since $d\g2=\tau_0\psi$, we can differentiate this to get $d\tau_0\wedge \psi =0$ and hence $d\tau_0 =0$, as wedge product with $\psi$ is an isomorphism from $\Omega^1_7(M)$ to $\Omega^5_7(M)$. Thus $\tau_0$ is a constant, if $M$ is connected.
\end{rem}

Given a $\G2$ structure $\g2$ with torsion $T_{lm}$, we have the expressions for the Ricci curvature $R_{ij}$ and the scalar curvature $S$ of its associated metric $g$ from \cite{kar05} as 

\begin{align}
R_{jk}&=(\del_iT_{jm}-\del_jT_{im})\g2_{nkl}g^{mn}g^{il} - T_{jl}g^{li}T_{ik}+ \text{Tr}(T)T_{jk}-T_{jb}T_{ia}g^{il}g^{ap}\psi_{lpqk}g^{bq} \label{ricci} \\
S&=-12g^{il}\del_i(\tau_1)_j+\frac{21}{8}{\tau_0}^2-|\tau_3|^2+5|\tau_1|^2-\frac 14 |\tau_2|^2 \label{scalar}
\end{align}
\noindent
where $|C|^2=C_{ij}C_{kl}g^{ik}g^{jl}$ is the matrix norm in \eqref{scalar}.

\vspace{0.2cm}

In particular, for a manifold $M$ with a nearly $\G2$ structure $\g2$, we see that 
\begin{align}
R_{jk} &= \frac{3}{8}{\tau_0}^2g_{jk} \label{nearlyricci}\\
S&=\frac{21}{8}{\tau_0}^2\label{nearlyscalar}
\end{align}

Finally, we remark that $S^7$ with the round metric and also the squashed $S^7$ are examples of manifolds with nearly $\G2$ structure (see \cite{friedrichnearlyg2} for more on nearly $\G2$ structures. The authors in \cite{friedrichnearlyg2} call such structures nearly parallel $\G2$ structures but we will call them nearly $\G2$ structures.) In particular, $S^7$ with radius $1$ has scalar curvature $42$, so comparing with \eqref{nearlyscalar} we get that $\tau_0 = 4$.

\subsection{Geometry of Submanifolds}
In this section, we briefly recall the geometry of submanifolds. More details can be found, for example in \cite{leebook}. Let $(\ol{M}, \ol{g})$ be Riemannian manifold and $(M,g)$ be an immersed orientable submanifold of $\ol{M}$ with induced metric. Then for $X,Y\in \Gamma(TM)$, we have 

\begin{equation}\label{gauss1}
\ol{\nabla}_XY=\nabla_XY+ \RNum{2}(X,Y)
\end{equation}
where $\ol{\nabla}$ is the covariant derivative on $\ol{M}$, $\nabla$ is the covariant derivative on $M$ and $\RNum{2}:TM\times TM \rightarrow NM$ is the second fundamental form of $M$. Here $NM$ is the normal bundle of $M$ in $\ol{M}$. 

If $M$ is an oriented hypersurface of $\ol{M}$ and we denote by $N$ the unit normal vector field of $M$ in $\ol{M}$ corresponding to this orientation, then the second fundamental form is a multiple of $N$ and is given by the $shape\ operator$, which we denote by $A$. Here $A:TM\rightarrow TM$ is a self-adjoint linear map and \eqref{gauss1} becomes

\begin{equation}\label{gauss2}
\ol{\nabla}_XY=\nabla_XY+ g(AX,Y)N
\end{equation}

We also have the Weingarten equation

\begin{equation}\label{weingarten}
\ol{\nabla}_XN = -AX
\end{equation}

If $\ol{Rm}$ denotes the Riemann curvature tensor on $(\ol{M}, \ol{g})$ and $Rm$ denotes the Riemann curvature tensor on $(M,g)$, then the Gauss equation for $M$ is

\begin{align}
\ol{Rm}(X,Y,Z,W) &= Rm(X,Y,Z,W)-g(AX,W)g(AY,Z)+g(AX,Z)g(AY,W) \label{gaussequation} 
\end{align}

Now suppose $\ol{M}$ is the unit sphere $S^7$ with the round metric. Then $\ol{Rm}$ as a $(3,1)-$tensor is given by $\ol{Rm}(X,Y)Z=\ol{g}(Y,Z)X-\ol{g}(X,Z)Y$. In this case \eqref{gaussequation} becomes 

\begin{align}\label{riem}
Rm(X,Y)Z&=\ol{g}(Y,Z)X-\ol{g}(X,Z)Y+g(AY,Z)AX-g(AX,Z)AY
\end{align}
If $M$ is also a \emph{minimal hypersurface} of $S^7$ (i.e., the mean curvature vector $H=0$ ) then by taking the trace of \eqref{riem}, the Ricci and the scalar curvature of $M$ are 

\begin{align}
Ric(X,Y)&=5g(X,Y)-g(AX,AY) \label{hyperricci}\\
S&=30-|A|^2 \label{hyperscalar}
\end{align}
where $|A|^2$ is the square of the length of the shape operator of $M$. We also have the Codazzi equation, which in this case is
\begin{equation}\label{codazzi}
\del_X(AY)-\del_Y(AX) = A([X,Y])
\end{equation}

Finally, we define totally umbilic hypersurface. 

\begin{defn}\label{tu}
A hypersurface $M$ of a Riemannian manifold $\ol{M}$ is called totally umbilic at $x\in M$ if the shape operator $A$ of $M$ is a multiple of the identity map of $T_xM$. Moreover $M$ is called totally umbilic if it is totally umbilic at each of its point.  
\end{defn}

\begin{rem}
Throughout the paper, all quantities associated to the ambient manifold $\ol{M}$ will have a bar with them, for example the metric on $\ol{M}$ is $\ol{g}$ whereas those of the hypersurface are written without any bar.
\end{rem}

\section{Proof of Theorem  \ref{mainthm1}}\label{sec3}
We start this section by defining various quantities for hypersurfaces (not necessarily minimal) of a manifold with a nearly $\G2$ structure which have analogs for hypersurfaces of a manifold with a nearly K\"ahler structure. Being motivated from the notion of a \emph{characteristic vector field} on a manifold with an almost complex structure,  we define a $(1,1)$ tensor $\xi$ on $M^6$, induced from the octonionic multiplication on a manifold with a $\G2$ structure $(\overline{M}^7, \g2)$, as follows

\begin{equation}\label{xi defn}
\xi(X) = B(N,X)
\end{equation}
where $X\in \Gamma(TM)$, $B(.,.)$ is the cross product and $N$ is the unit normal to $M^6$ in $\overline{M}$. We have the following 

\begin{prop}\label{prop3.1}
The tensor $\xi$ is a metric compatible almost complex structure on $(M^6, g)$.
\end{prop}
\begin{proof}
For $X\in \Gamma(TM)$, we have 
\begin{align*}
\xi^2(X) &= \xi(B(N,X)) = B(N,B(N,X)) \\
&= -|N|^2X+\ol{g}(N,X) N = -X
\end{align*}
where the equality in the second line is from \eqref{malcev} for cross product. Hence $\xi^2(X)=-X$. Also,

\begin{align*}
g(\xi(X), \xi(Y))  &= g( B(N,X), B(N,Y)) \\
&= g(B(B(N,X), N), Y) \\
&=-g( B(N, B(N,X)), Y ) = g (X, Y) 
\end{align*}
 where we have used \eqref{cp1} in going from the first to the second line, the anti-commutativity of $B$ in the first equality and \eqref{malcev} and the fact that $N$ is a unit vector in the second equality of the third line.

\end{proof}

\begin{rem}
The previous proposition is a special case of Theorem 2.6 in \cite{gray-VCP}.
\end{rem}

Again, from the motivation from nearly K\"ahler geometry, we define a $(3,1)$ tensor field $G$ as follows
\begin{equation}\label{G defn}
G(X,Y,Z) = (\ol{\del}_XB)(Y,Z)
\end{equation}
for $X,Y,Z\in \Gamma(T\ol{M})$.

Now we prove some results about $G$ and relationships between $G$ and $B$ for manifolds with a nearly $\G2$ structure. The next proposition is a special case of Lemma 3.7 in \cite{semmelman}.
 
 \begin{prop}\label{prop3.2}
 Let $\psi = *\g2$ denotes the $4$-form on $(\ol{M}, \g2)$ with a nearly $\G2$ structure. Then for any vector fields $X,Y,Z,W$
 \begin{equation}\label{Gpsi}
 \ol g (G(X,Y,Z), W) = \frac{\tau_0}{4}\psi(X,Y,Z,W)
 \end{equation}
 where $\tau_0$ is as defined in \eqref{torsionforms}. 
 
 \end{prop}
 
 \begin{proof}
 If $\g2$ is a $\G2$ structure then 
 \begin{equation}\label{phi and B}
 \g2(X,Y,Z) = \ol g( B(X,Y), Z)
 \end{equation}

 Then from \eqref{phi and B} we have 
 \begin{align*}
 \ol g( G(X,Y,Z), W)&=\ol g ((\ol\del_XB)(Y,Z), W) \\
 &= \ol g( \ol\del_X(B(Y,Z)) - B(\ol\del_XY,Z)-B(Y,\ol\del_XZ), W)\\
 &= \ol\del_X(\g2(Y,Z,W)) - \g2(\ol\del_XY, Z,W) - \g2(Y, \ol\del_XZ,W)\\
 & \quad -\g2(Y,Z,\ol\del_XW)\\
 &= (\ol\del_X\g2)(Y,Z,W)\\
 &=\frac{\tau_0}{4}\psi(X,Y,Z,W)
 \end{align*}
 where we have used  $\ol g( \ol{\del}_X(B(Y,Z)), W) = \ol{\del}_X(\ol g(B(Y,Z), W)) - \ol g( B(Y,Z), \ol\del_X(W))$ in going from the second to the third equality, $\del_i\g2_{jkl} = T_{im}g^{mp}\psi_{pjkl}$ and the fact that for a nearly $\G2$ structure, $T_{ij}=\frac{\tau_0}{4}g_{ij}$ in the last equality.
 
 \end{proof}
 
 \begin{rem}\label{Gremark}
From \eqref{Gpsi}, we see that $G$ is skew-symmetric in all of its entries.
\end{rem}

\begin{prop}\label{prop3.3}
For any vector fields $X,Y,Z,W$, we have 
\begin{align}\label{G and B}
G(B(W,Z), X,Y) &= \frac{\tau_0}{4}[\ol g(X,Z)B(W,Y)+ \ol g(Y,Z)B(X,W)-\ol g(W,X)B(Z,Y)  \nonumber \\ 
&\quad -\ol g(W,Y)B(X,Z) +\g2(X,Y,W)Z-\g2(X,Y,Z)W ] 
\end{align}

\end{prop}

\begin{proof}
We know from Proposition \ref{prop3.2} that 
\begin{equation*}
G(X,Y,Z)= \frac{\tau_0}{4}\psi(X,Y,Z, \cdot)
\end{equation*}
so 
\begin{equation*}
G(B(X,Y), Z, W)=\frac{\tau_0}{4}\psi(B(X,Y), Z,W, \cdot)^{\#}
\end{equation*}

In local coordinates $\{x_1,x_2,\ldots ,x_7\}$, we have $\ol g(B(\pt_k,\pt_l),\pt_n)=\g2_{kln}$. So

\begin{align}
G(B(\pt_k, \pt_l), \pt_i,\pt_j) &= \frac{\tau_0}{4}\psi(B(\pt_k,\pt_l),\pt_i,\pt_j, \cdot )^{\#} \nonumber \\
&=\frac{\tau_0}{4}\psi({\g2_{kl\cdot}}^{\#}, \pt_i,\pt_j, .)^{\#} \nonumber \\
&=-\frac{\tau_0}{4}\psi_{ij\cdot n}\ol{g}^{np}\g2_{klp} \label{aux1}
\end{align}

Using the identity in \eqref{contractions2}  

\begin{equation*}
\ol{g}^{np}\psi_{ijmn}\g2_{klp} = \ol g_{ki}\g2_{ljm}+\ol g_{kj}\g2_{ilm}+\ol g_{km}\g2_{ijl}-\ol g_{li}\g2_{kjm}-\ol g_{lj}\g2_{ikm}-\ol g_{lm}\g2_{ijk}
\end{equation*}
\vspace{0.15cm}
we get the proposition. 
\end{proof}

\begin{prop}\label{prop3.4}
For any vector fields $X,Y,Z,W$, we have
\begin{equation}\label{B and G}
B(G(X,Y,Z),W) = -G(B(X,Y), Z, W)
\end{equation}
\end{prop}
\begin{proof}
In local coordinates $\{x_1,\dots, x_7\}$, we have $G(\pt_k,\pt_l,\pt_m)=\frac{\tau_0}{4}\psi_{klm\cdot}^{\#}$ and $B(\pt_k,\pt_l)=\g2_{kl\cdot}^{\#}$, so

\vspace{0.1cm}

\begin{equation*}
B(G(\pt_k,\pt_l,\pt_m),\pt_n)=\frac{\tau_0}{4}\g2(\psi_{klm\cdot}^{\#}, \pt_n,\cdot)^{\#}=\frac{\tau_0}{4}\g2_{n\cdot p}^{\#}\ol{g}^{ps}\psi_{klms}
\end{equation*}

The proposition now follows from the last line of \eqref{aux1}.

\end{proof}

We will need the expression for $\del_X\xi$ later, so we have the following proposition which is a special case of Proposition 4.7 of \cite{gray-VCP} for the $2$-fold VCP.

\begin{prop}\label{prop3.5}
Let $M$ be an oriented hypersurface of $(\ol{M}, \g2)$ and $\xi$ be as defined in \eqref{xi defn}. Then for any vector field $X\in \Gamma(TM)$, we have 

\begin{equation}\label{delxi defn}
(\del_X \xi)(Y)=G(X,N,Y)-\g2(N,Y,AX)N-B(AX,Y)
\end{equation}
\end{prop}

\begin{proof}
We calculate
\begin{align}
(\del_X\xi)(Y)&= \del_X(\xi(Y))-\xi(\del_XY) \nonumber \\
&=\ol\del_X(B(N,Y))-g(AX, B(N,Y))N-\xi(\del_XY)\nonumber \\
&=(\ol\del_X B)(N,Y)+B(\ol\del_XN,Y)+B(N,\ol\del_XY)-g(AX,B(N,Y))N \nonumber\\
& \quad -\xi(\del_XY) \nonumber \\
&= G(X,N,Y)-B(AX, Y)+B(N,\del_XY)+g(AX, Y)B(N,N) \nonumber \\
& \quad -g(AX, B(N,Y))N-\xi(\del_XY) \nonumber \\
&=G(X,N,Y)-\g2(N,Y,AX)N-B(AX,Y)
\end{align}
where we have used \eqref{gauss2} in the second equality, \eqref{weingarten} and \eqref{G defn} in the fourth equality and the fact that $B(N,N)=0$ in the last equality.
\end{proof}

Now we will prove Theorem \ref{mainthm1} mentioned in \cref{sec-intro}, namely, we will give a necessary and sufficient condition for an oriented hypersurface of a nearly $\G2$ manifold to be nearly K\"ahler. We restate the theorem.

\begin{thm}\label{maintheorem1}
Let $M$ be an oriented hypersurface of a nearly $\G2$ manifold $(\ol{M},\g2)$. Then $(M,g,\xi)$ is a nearly K\"ahler structure if and only if $M$ is totally umbilic, i.e., for all $X\in \Gamma(TM)$
\begin{equation}
AX=\alpha X
\end{equation}
where $A$ is the shape operator of $M$ in $\ol{M}$ and $\alpha \in C^{\infty}(M)$.
\end{thm}
\begin{proof}
We know from \eqref{nearlykahlercondition} that if $J$ is a metric compatible almost complex structure on $M$ then $(M,J,g)$ is nearly K\"ahler if and only if for all $X\in \Gamma(TM)$, we have $(\del_XJ)X=0$. From Proposition \ref{prop3.1}, we know that $\xi$ is a metric compatible almost complex structure on $M$. Denote by $B(X,Y)^T$, the tangential component of $B(X,Y)$. Using \eqref{delxi defn} from Proposition \ref{prop3.5}, for $X\in \Gamma(TM)$
\begin{align}
(\del_X \xi)(X)&=0 \iff \nonumber \\
G(X,N,X)-\g2(N,X,AX)N-B(AX,X)&=0 \iff \nonumber \\
\g2(N,X,AX)N+B(AX,X)^T+g(B(AX,X), N)N&=0 \iff \nonumber \\
B(AX,X)^T+\g2(AX,X,N)N+\g2(N,X,AX)N&=0\iff \nonumber \\
B(AX,X)^T&=0\label{bt}
\end{align}
where we used the fact that $G$ is skew-symmetric in all of its entries in going from the second line to the third. 

If $X=0$ then from \eqref{bt}, the theorem is true. So we assume that $X\neq 0$. Now if $AX=\alpha X$ then
\begin{align}
B(AX,X)^T&=B(\alpha X, X)^T \nonumber \\
&=0\label{bt1}
\end{align}
 Thus \eqref{bt} and \eqref{bt1} proves one direction of the theorem.

\vspace{0.1cm}

Now suppose $B(AX,X)^T=0$. Since $AX$ is tangent to $M$ so we write $AX=\alpha X+Y$ where $\alpha$ is a function which might depend on $X$ and $g(X,Y)=0$. So $B(Y,X)^T=0$. Suppose 
\begin{align*}
B(Y,X)=aN
\end{align*}
for some function $a$.  

Then from \eqref{malcev} we have
\begin{align*}
B(B(Y,X),X)&=-|X|^2Y
\end{align*}

Also, $B(B(Y,X),X)=aB(N,X)=a\xi(X)$, so we get
\begin{align*}
Y=-\frac{a}{|X|^2}\xi(X)
\end{align*}
and hence 
\begin{align}
AX=\alpha X +\beta \xi(X)\label{bt2}
\end{align}
where $\beta = -\frac{a}{|X|^2}$. Now we prove that $\beta=0$. Indeed, for any $Z\in \Gamma(TM)$,
\begin{align*}
g(AX,Z)&=g(\alpha X + \beta \xi (X),Z) \\
&=\alpha g(X,Z)-\beta g(X, \xi (Z))
\end{align*}
and similarly $g(X,AZ)=\alpha g(X,Z)+\beta g(X, \xi (Z))$. But since $A$ is self-adjoint we get that $2\beta g(X, \xi (Z))=0$. So choosing $Z$ such that $B(X,Z)=N$, we get that $\beta=0$. 
This proves the other direction.
 \end{proof}
 
 \begin{rem}
 Note that the proof of Theorem \ref{maintheorem1} remains unchanged if $G=0$. So the above theorem also holds for hypersurfaces of $\G2$ manifolds, i.e., manifolds with torsion free $\G2$ structures. 
 \end{rem}
 
 \begin{rem}
Theorem \ref{maintheorem1} was proved in \cite[Theorem 4.8]{gray-VCP} where Gray proved that $\beta=0$ by using the fact that any nearly K\"ahler structure $J$ is \emph{quasi-K\"ahler}, i.e., for all $X,Y\in \Gamma(TM),\ (\del_XJ)(Y)+(\del_{JX}J)(JY)=0$. We gave a direct proof that $\beta=0$. 
 \end{rem}
 
 \begin{rem}
 Koiso proved in \cite[Theorem B]{koiso} that if $(M,g)$ is a totally umbilic Einstein hypersurface in a complete Einstein manifold $(\ol M, \ol g)$ and $g$ have positive Ricci curvature then both $g$ and $\ol g$ have constant sectional curvature. This restricts the possibility for new examples of hypersurfaces which are totally umbilic. It would be interesting to find examples of hypersurfaces in a manifold with a nearly $\G2$ structure which are nearly K\"ahler with respect to $\xi$ but are not totally umbilic.
 \end{rem}
 We will need the following Lemma in \cref {sec-thm2} which is a special case of Theorem 4.10 in \cite{gray-VCP}. 
 
 \begin{lemma}\label{divxi}
 Let $M$ be an oriented hypersurface of a nearly $\G2$ manifold $(\ol{M}, \g2)$ and let $\xi$ be as in \eqref{xi defn}. Then $\Div \xi = 0$.
 \end{lemma}
 \begin{proof}
 Since $A$ is a self-adjoint operator, so we choose an orthonormal frame $\{e_1,...,e_6\}$ at a point $p\in M$ which diagonalizes $A$, i.e., $Ae_i = a_ie_i$, $\forall i$. Then for $v\in T_pM$ we compute using Proposition \ref{prop3.5}
 \begin{align}
 (\Div \xi)_p(v) &= \sum_{i=1}^6 g((\del_{e_i}\xi)(v), e_i) \nonumber \\
 &= \sum_{i=1}^6 g(G(e_i, N,v)-\g2(N,v,Ae_i)N-B(Ae_i, v),e_i) \nonumber \\
 &=  -\sum_{i=1}^6g(B(Ae_i, v),e_i) = - \sum_{i=1}^6 \g2(Ae_i,v,e_i)=-\sum_{i=1}^6 \g2(a_ie_i, v, e_i) \nonumber \\
 &=0
 \end{align}
where we used \eqref{delxi defn} in the second equality, Remark \ref{Gremark} in the third equality and the fact that $\g2$ is $3$-from in the last equality.
 \end{proof}

\section{Proof of Theorem \ref{mainthm2}}\label{sec-thm2}

In this section we will prove Theorem \ref{mainthm2}, stated in \cref{sec-intro}. Let $(L,g)$ be a Riemannian manifold. A vector field $X$ on $L$ is said to be a \emph{conformal vector field} if 

\begin{equation}\label{conformal}
\mathcal{L}_Xg=2fg
\end{equation}
for some $f\in C^{\infty}(L)$, which is called the \emph{potential} of $X$. Here $\mathcal{L}_Xg$ denotes the Lie derivative of $g$ with respect to $X$. If $f\equiv 0$, then $X$ is a Killing vector field. There are many non-Killing conformal vector fields on the unit sphere $S^n$ with the round metric $\ol g$. In particular, if $Y$ is a non-zero constant vector field on $\mathbb{R}^{n+1}$, $\ol{N}$ is the unit normal of $S^n$ in $\mathbb{R}^{n+1}$ and $Y=X+f\ol{N}$, where $X$ is the tangential component of $Y$, then using \eqref{gauss2} and \eqref{weingarten} and the fact that for $S^n$ as a hypersurface in $\mathbb{R}^{n+1}$, $A=-I$, we see that $\del f=X$ and $\del_WX=-fW$, and hence $\mathcal{L}_X\ol g=-2f\ol g$, so $X$ is a conformal vector field with potential $-f$. In fact, all non-Killing conformal vector fields on the unit $S^n$ arise in this manner. (see \cite{moques})

\vspace{0.2cm}

 Let $M$ be an oriented compact minimal hypersurface of $S^7$ satisfying the hypotheses of Theorem \ref{mainthm2}, i.e., $M$ is of constant scalar curvature and the shape operator $A$ of $M$ satisfies $|A|^2>6$. Let $V,\td{V}$ be two non-Killing conformal vector fields on $S^7$ with potential functions $f, \td f$ respectively, arising from two linearly independent constant vector fields on $\mathbb{R}^8$. Let $W, \td W$ be the tangential components on $M$ of $V$ and $\td V$ respectively. Then we have $V=W+sN$ and $\td V=\td W +\td s N$, where $s, \td s:M\rightarrow \mathbb{R}$.

\vspace{0.2cm}

Using \eqref{gauss2} and \eqref{weingarten}, for $X\in \Gamma(TM)$ we get 

\begin{align}\label{wderivative}
\del_XW &= \ol\del_XV-\ol\del_X(sN) \nonumber \\
&=-fX+sAX \\
\del f & = W \label{fderivative} \\
\del s & = -AW \label{sderivative} 
\end{align}

Similarly, we get 
\begin{equation}\label{wtdderivative}
\del_X\td W = -\td fX+\td sAX , \qquad \del \td f = \td W \quad \text{and} \quad  \del \td s = -A\td W
\end{equation}

Now we define the function $h:M\rightarrow \mathbb{R}$ as 

\begin{equation}\label{hdefn}
h=g(\xi W, \td W)
\end{equation}

We are interested in finding $\Delta_M h$. So we compute

\begin{align}
\del_Xh &=\del_Xg(\xi W,\td W) \nonumber \\
&= g((\del_X\xi)W, \td W)+g(\xi(\del_X W), \td W)+g(\xi W, \del_X \td W) \nonumber \\
&=g\big(G(X,N,W)-\g2(N,W,AX)N-B(AX,W)^T,\td W\big) \nonumber \\
& \quad +g(\xi(-fX+sAX), \td W)+g(\xi W, -\td fX+\td sAX) \nonumber \\
& = -g(G(N, W, \td W),X)-g(B(W, \td W)^T, AX)+g(f\xi \td W, X) \nonumber \\
& \quad -g(s\xi \td W, AX)-g(\td f\xi W, X)+g(\td s \xi W, AX) \nonumber \\
\end{align}
so we get

\begin{align}
\del h &= -G(N, W, \td W)-AB(W, \td W)^T+f\xi \td W-sA\xi \td W - \td f \xi W + \td s A\xi W\label{delh}
\end{align}

\vspace{0.3cm}

We use \eqref{wderivative}, \eqref{fderivative}, \eqref{sderivative} and \eqref{wtdderivative} to calculate the divergence of each term in \eqref{delh}. For that, we choose local orthonormal frame $\{e_1,...,e_6\}$ at $p\in M$ such that $Ae_i = a_ie_i, \ \forall \ i$.

\begin{align}
\Div (f\xi \td W)&=  g(\del f , \xi \td W)+f\sum_{i=1}^6[g((\del_i \xi)\td W, e_i)+g(\xi(\del_i\td W), e_i)] \nonumber \\
&=g(W, \xi \td W) + \sum_{i=1}^6 g(\xi(-\td fe_i+\td sAe_i),e_i) \nonumber \\
&=g(W, \xi \td W) + \sum_{i=1}^6[-fg(\xi e_i,e_i)+\td{s}a_ig(\xi e_i,e_i)]\nonumber \\
&=-h \label{div1}
\end{align}
where we have used Lemma \ref{divxi} in the second equality and the definition of $\xi$ to eliminate the terms inside the summation in the third equality. 

\vspace{0.3cm}

Similarly
\begin{align}\label{div2}
\Div(\td f \xi W)=h
\end{align}

\begin{align}
\Div(sA\xi \td W)&= g(\del s, A\xi \td W)+s\sum_{i=1}^6g(\del_i(A\xi \td W), e_i)\nonumber \\
&=-g(AW, A\xi \td W)+s\sum_{i=1}^6[\del_ig(A\xi \td{W},e_i)-g(A\xi \td{W}, \del_{e_i}e_i)] \nonumber \\
&=-g(AW, A\xi \td W)+s\sum_{i=1}^6[g(\del_i\xi \td{W}, Ae_i)+g(\xi \td{W}, \del_{e_i}Ae_i) \nonumber \\
& \qquad \qquad \qquad \qquad \qquad -g(\xi \td{W}, A\del_{e_i}e_i)]\nonumber \\ 
&=-g(AW, A\xi \td W)+s\sum_{i=1}^6[g((\del_{e_i}\xi)\td W,Ae_i)+g(\xi(\del_{e_i}\td W), Ae_i)]\nonumber \\  
&=-g(AW, A\xi \td W) \label{div3}  
\end{align}

\vspace{0.15cm}

\noindent
where in the third equality we have used that $\sum_i (\nabla A)(e_i,e_i)=\sum_i(\del_{e_i}Ae_i-A\del_{e_i}e_i)=0 $ which follows from the Codazzi identity \eqref{codazzi} and the fact that $M$ is minimal, \eqref{delxi defn}, \eqref{wtdderivative} and $Ae_i=a_ie_i$ to eliminate the terms inside the summation in the second last equality.

\vspace{0.3cm}

Similarly 
\begin{align}\label{div4}
\Div(\td sA\xi W)=-g(A\td W, A\xi W)
\end{align}

\vspace{0.4cm}

For calculating $\Div (AB(W, \td W)^T)$, we repeatedly use \eqref{gauss2} and \eqref{weingarten} to first compute

\begin{align}
(\ol \del_ZB)(X,Y)&=\ol \del_Z(B(X,Y))-B(\ol \del_ZX, Y)-B(X, \ol \del_Z Y) \nonumber \\
&= \ol \del_Z\big(B(X,Y)^T+\ol g(B(X,Y), N)N\big)-B(\del_ZX, Y)\nonumber \\
& \quad -g(AZ, X)B(N,Y)-B(X, \del_ZY)-g(AZ,Y)B(X,N) \nonumber \\
&= \del_Z(B(X,Y)^T)+g(AZ, B(X,Y)^T)N +\ol g(\ol \del_Z(B(X,Y)), N)N \nonumber \\
& \quad -\ol g(B(X,Y), AZ)N - \ol g(B(X,Y), N)AZ-B(\del_ZX, Y) \nonumber \\
& \quad -g(AZ, X)B(N,Y) -B(X, \del_ZY)-g(AZ,Y)B(X,N) \nonumber \\
\end{align}

\vspace{0.15cm}

\noindent
where we have written $B(X,Y)$ as a sum of its tangential and normal components in the first term in second equality and then used \eqref{gauss2} in the third equality. So we get

\begin{align}
(\ol \del_ZB)(X,Y)&=\del_Z(B(X,Y)^T)+\ol g((\ol \del_ZB)(X,Y), N)N+\ol g(B(\del_ZX, Y), N)N \nonumber \\
& \quad +g(AZ,X)\ol g(B(N,Y), N)N +\ol g(B(X, \del_ZY), N)N \nonumber \\
& \quad +g(AZ,Y)\ol g(B(X,N),N)N - \ol g(B(X,Y), N)AZ-B(\del_ZX, Y) \nonumber \\
& \quad -g(AZ, X)B(N,Y) -B(X, \del_ZY)-g(AZ,Y)B(X,N) \nonumber \\
&=\del_Z(B(X,Y)^T)+\ol g((\ol \del_ZB)(X,Y), N)N+\ol g(B(\del_ZX, Y), N)N \nonumber \\
&\quad +\ol g(B(X, \del_ZY), N)N- \ol g(B(X,Y), N)AZ-B(\del_ZX, Y) \nonumber \\
& \quad -B(X, \del_ZY)-g(AZ, X)B(N,Y) -g(AZ,Y)B(X,N) \label{div5}
\end{align}

\vspace{0.25cm}

\noindent
where we have used $\ol g(B(N, V),N)=\g2(N,V,N)=0$, $\forall \ V$ in going from fourth to fifth equality. Now using \eqref{G defn}, we see that \eqref{div5} is

\begin{align}
\del_Z(B(X,Y)^T)&=G(Z,X,Y)^T-\ol g(B(\del_ZX, Y), N)N-\ol g(B(X, \del_ZY), N)N \nonumber \\
& \quad +\ol g(B(X,Y), N)AZ+B(\del_ZX, Y) +B(X, \del_ZY) \nonumber \\
& \quad +g(AZ, X)B(N,Y)+g(AZ,Y)B(X,N)\label{div6}
\end{align}

\vspace{0.3cm}

Using \eqref{div6} we calculate

\begin{align}
\Div(AB(W,\td W)^T)&=\sum_{i=1}^6 [g((\del_iA)(B(W,\td W)^T),e_i)+g(\del_{e_i}(B(W, \td W)^T),Ae_i)] \nonumber \\
&= \sum_{i=1}^6 g\big ((G(e_i,W,\td W)^T-\ol g(B(\del_{e_i}W, \td W), N)N \nonumber \\
& \quad -\ol g(B(W, \del_{e_i}\td W), N)N +\ol g(B(W,\td W), N)Ae_i +B(\del_{e_i}W, \td W) \nonumber \\ 
&  \quad +B(W, \del_{e_i}\td W)+g(Ae_i, W)B(N,\td W)+g(Ae_i,\td W)B(W,N)), Ae_i\big ) \nonumber \\
&=\sum_{i=1}^6 [\ol g (B(N,W), \td W)g(Ae_i, Ae_i)+g(B(-fe_i+sAe_i, \td W), Ae_i)\nonumber \\
& \quad +g(B(W, -\td fe_i+\td s Ae_i), Ae_i)+g(e_i, AW)g(B(N,\td W), Ae_i) \nonumber \\
& \quad +g(e_i, A\td W)g(B(W,N), Ae_i)] \nonumber \\
&= |A|^2h+g(AW, A\xi \td W)-g(A\td W, A\xi W)\label{div7}
\end{align}

\medskip

\noindent
where we have used Remark \ref{Gremark} to eliminate the first term inside the summation in the second equality, Proposition \ref{prop3.2} in the third equality and the facts that $\ol g(B(a,b),c)=\g2(a,b,c)$ and $Ae_i=a_ie_i$ in going from the third to last equality.

\vspace{0.3cm}

For calculating $\Div(G(N,W, \td W))$, we first of all note that due to Proposition \ref{prop3.2}, $G(N,X,Y)$ is tangent to $M$ for any $X, Y\in \Gamma(TM)$. We calculate

\begin{align}
\Div(G(N,W, \td W))&=\sum_{i=1}^6 g(\del_i(G(N,W,\td W)), e_i) \nonumber \\
&=\sum_{i=1}^6 [\del_i(g(G(N,W, \td W), e_i))-g(G(N,W,\td W), \del_ie_i)] \nonumber \\
&=\frac{\tau_0}{4}\sum_{i=1}^6[(\del_i\psi)(N,W,\td W,e_i)+\psi(\del_iN,W,\td W, e_i)+\psi(N,\del_iW, \td W,e_i) \nonumber \\
& \quad+\psi(N,W,\del_i \td W,e_i)+\psi(N,W,\td W, \del_ie_i)-\psi(N,W,\td W, \del_ie_i)]\nonumber \\
&=\frac{\tau_0}{4}\sum_{i=1}^6[(\del_i\psi)(N,W,\td W,e_i)-\psi(Ae_i,W,\td W,e_i)]\nonumber \\
&=\frac{\tau_0}{4}\sum_{i=1}^6[\frac{\tau_0}{4}(-g(e_i,N)\g2(W,\td W,e_i)+g(e_i,W)\g2(N,\td W,e_i) \nonumber \\
& \quad-g(e_i,\td W)\g2(N,W,e_i)+g(e_i,e_i)\g2(N,W,\td W))]\nonumber \\
&=\frac{{\tau_0}^2}{16}[g(\xi \td{W}, W)-g(\xi W, \td{W})+\sum_{i=1}^6g(e_i,e_i)g(\xi W, \td W)]\nonumber \\
&=\frac{{\tau_0}^2}{4} h \label{div8}
\end{align}

\medskip

\noindent
where we have used Proposition \ref{prop3.2} in the third equality, \eqref{wderivative}, \eqref{wtdderivative}, $Ae_i=a_ie_i$ and the fact that $\psi$ is a $4$-form to eliminate the $\psi(N,\del_iW,\td{W},e_i)$ and $\psi(N,W,\del_i\td{W},e_i)$ in the third equality, \eqref{torsion3} (expression for $\del_i\psi_{jklm}$) and the fact that for nearly $\G2$ structures $T_{ij}=\frac{\tau_0}{4}g_{ij}$ in the fifth equality.

\vspace{0.2cm}

Using the fact that for the unit $S^7$, $\tau_0=4$, \eqref{delh}, \eqref{div1}, \eqref{div2}, \eqref{div3}, \eqref{div4}, \eqref{div7} and \eqref{div8} we see that

\begin{align}
\Delta_Mh&=-4h-|A|^2h-g(AW,A\xi \td{W})+g(A\td{W}, A\xi W)-h+g(AW, A\xi \td{W})\nonumber \\
& \quad -h-g(A\td{W}, A\xi W)
\end{align}
so
\begin{equation}\label{laplaceh}
\Delta_Mh = -(|A|^2+6)h
\end{equation}

\vspace{0.2cm}

Now if $h$ is a constant function then \eqref{laplaceh} implies that $h=0$, i.e., $g(\xi(W), \td W)=0$. Recall that $\td W$ is the tangential component of a non-Killing conformal vector field $\td V$ on $S^7$ where $\td{V}$ is the tangential component of any constant vector field on $\mathbb{R}^8$. The vector field $W$ was obtained in a similar manner by taking the tangential component of a non-Killing conformal vector field $V$ on $S^7$ which was obtained as the tangential component of a constant vector field on $\mathbb{R}^8$ which was linearly independent from the constant vector field which gives $\td{V}$. So if $g(\xi(W), \td{W})=0$ for all $\td{W}$,  we get $\xi(W)=0$, i.e., $B(N,W)=0$. This is a contradiction because $\xi$ is invertible and $N$ is a unit vector. Hence there exists $W,\ \td{W}$ such that $h$ is not constant and \eqref{laplaceh} implies that $h$ is an eigenfunction of $\Delta_M$ corresponding to the eigenvalue $\lambda = |A|^2+6$. So if $|A|^2>6$ then $\lambda>12$ with $|A|^2=\lambda -6$. The proof of Theorem \ref{mainthm2} is now complete.

\nocite{Kar03} \nocite{lee-leung} \nocite{lotay-wei1}

\bibliographystyle{amsplain}
\bibliography{Readinglist}

\end{document}